\documentclass[a4paper,10pt]{article}
\usepackage[T1]{fontenc}
\usepackage[utf8]{inputenc}
\usepackage[english,greek,frenchb,french]{babel}
\usepackage{url,csquotes}
\usepackage[hidelinks,hyperfootnotes=false]{hyperref}
\usepackage{float}
\usepackage{amsfonts,amssymb,enumerate}
\usepackage{amsthm}
\usepackage{amsmath}
\usepackage{mathtools,thmtools}
\allowdisplaybreaks[1]
\usepackage{graphicx}
\usepackage{bbm}
\usepackage{listings}
\usepackage{titling}
\usepackage{dsfont}
\usepackage{color}

\newcommand{\E}{\mathbb{E}}

    \newcommand{\Prb}{\mathbb{P}}
	
		\newcommand{\cG}{\mathcal{G}}
				\newcommand{\cA}{\mathcal{A}}

		\newcommand{\cH}{\mathcal{H}}
		\newcommand{\cI}{\mathcal{I}}
		
		\newcommand{\cL}{\mathcal{L}}
	
	\newcommand{\sR}{\mathbb{R}}
	\newcommand{\sN}{\mathbb{N}}
		\newcommand{\sS}{\mathbb{S}}
		
    \newcommand{\sZ}{\mathbb{Z}}
    \newcommand{\sC}{\mathcal{C}}
        \newcommand{\sA}{\mathcal{A}}

\theoremstyle{plain}
\newtheorem{thm}{Theorem}[section]
\newtheorem{lem}[thm]{Lemma}

\newlength{\separationtitre}
\author{ Rapha\"el Cerf\thanks{
DMA, Ecole Normale Sup\'erieure,
CNRS, PSL University, 75005 Paris.}
\thanks{LMO, Universit\'e Paris-Sud, CNRS, Universit\'e
Paris--Saclay, 91405 Orsay.}, Barbara Dembin\thanks{LPSM UMR 8001, Université Paris Diderot, Sorbonne Paris Cité, CNRS, F-75013 Paris.}\thanks{The author is laureate of the Séphora Berrebi Scholarship in Mathematics in 2019. The author would like to thank Association Séphora Berrebi for their support. }}
\date{}
\begin{document}
\newpage

 \selectlanguage{english}

\title{Vanishing of the anchored isoperimetric profile in bond percolation at $p_c$ 
}

\maketitle

\begin{abstract}
We consider the anchored isoperimetric profile of the 
infinite open cluster, defined for $p>p_c$, whose existence has been recently proved in \cite{DembinCheeger}. 
We extend adequately the definition for $p=p_c$, in finite boxes. We prove a partial result which implies that, if the limit defining the anchored isoperimetric profile at $p_c$ exists, it has to vanish.
\end{abstract}

\section{Introduction}

The most well--known open question in percolation theory is to prove that the percolation probability vanishes at $p_c$ in dimension three.
In fact, the interesting quantities associated to the model are very difficult to study at the critical point or in its vicinity. 
We study here a very modest intermediate question. We consider the anchored isoperimetric profile of the infinite open cluster, defined for $p>p_c$, whose existence has been recently proved in \cite{DembinCheeger}. We extend adequately the definition for $p=p_c$, in finite boxes. We prove a partial result which implies that, if the limit defining the anchored isoperimetric profile at $p_c$ exists, it has to vanish.
\vskip 0.5cm
\noindent
{\bf The Cheeger constant.}
For a graph $\cG$ with vertex set $V$ and edge set $E$, we define the edge boundary $\partial_\cG A$  of a subset $A$ of $V$ as  $$\partial_\cG A =\Big\{\,e=\langle x,y\rangle \in E:x\in A,y\notin A \,\Big\}\,. $$ 
We denote by $|B|$ the cardinal of the finite set $B$. The Cheeger constant of the graph $\cG$ is defined as
$$\varphi_\cG=\min\left\{\,\frac{|\partial_\cG A|}{|A|}\,: \, A\subset V, 0<|A|\leq \frac{|V|}{2}\,\right\}\,.$$
This constant was introduced by Cheeger in his thesis \cite{thesis:cheeger}  in order to obtain a lower bound for the smallest eigenvalue of the Laplacian. 
\vskip 0.5cm
\noindent
{\bf The 
anchored isoperimetric profile  $\varphi_n(p)$.}
Let $d\geq 2$. We consider an i.i.d. supercritical bond percolation on $\sZ^d$, every edge is open with a probability $p>p_c(d)$, where $p_c(d)$ denotes the critical parameter for this percolation. We know that there exists almost surely a unique infinite open cluster $\sC_\infty$ \cite{Grimmett99}. 
We say that $H$ is a valid subgraph of $\sC_\infty$
if $H$ is connected and $0\in H\subset \sC_\infty$.
We define the anchored isoperimetric profile  $\varphi_n(p)$ of $\sC_\infty$ as follows.
We condition
on the event $\{0\in\sC_\infty\}$ and we set
$$ \varphi_n(p)=\min\left\{\,\frac{|\partial_{\sC_\infty} H|}{|H|}: 
\text{ $H$ valid subgraph of $\sC_\infty$},  
\, 0<|H|\leq n^d\,\right\}
\,. $$
The following theorem 
from \cite{DembinCheeger} 
asserts the existence of the limit of $n\varphi_n(p)$ when $p>p_c(d)$. 
\begin{thm}\label{thmheart}
Let $d\geq 2$ and $p>p_c(d)$.  
	There exists a positive real number $\varphi(p)$ such that, conditionally on $\{0\in\sC_\infty\}$,
$$\lim_{n\rightarrow \infty} n\varphi_n(p)\,=\,\varphi(p)\text{ almost surely.}$$
\end{thm}
\noindent
We wish to study how this limit behaves when $p$ is getting closer to $p_c$. To do so, we need to extend the definition of the anchored isoperimetric profile so that it is well defined at $p_c(d)$.
We say that $H$ is a valid subgraph of $\sC(0)$,
the open cluster of $0$, 
if $H$ is connected and $0\in H\subset \sC(0)$.
We define $\widehat{\varphi}_n(p)$ for every $p\in[0,1]$ as
$$\widehat{\varphi}_n(p) = 
        \min\left\{\,\frac{|\partial_{\sC(0)} H|}{|H|}:
\text{ $H$ valid subgraph of $\sC(0)$},  
\, 0<|H|\leq n^d\,\right\}
\,. $$
In particular, if $0$ is not connected to $\partial [-n/2,n/2]^d $ by a $p$-open path, then $|\sC(0)|<n^d$ and taking $H=\sC(0)$, we see that 
$\widehat{\varphi}_n(p)$ is equal to $0$.
Thanks to theorem \ref{thmheart}, we have 
$$\forall p>p_c\qquad \lim_{n\rightarrow \infty}n\widehat{\varphi}_n(p)\,=\,\theta(p)\delta_{\varphi(p)}+(1-\theta(p))\delta_0\,,$$
where $\theta(p)$ is the probability that $0$ belongs to an infinite open cluster.
The techniques of~\cite{DembinCheeger} to prove the existence of this limit rely on coarse--graining estimates which can be employed only
in the supercritical regime. Therefore we are not able so far to extend the above convergence at the critical point $p_c$.
Naturally, we expect that $n \widehat{\varphi}_n(p_c)$ converges towards $0$ as $n$ goes to infinity, unfortunately we are only able to prove a weaker statement.
\begin{thm}\label{heartp} With probability one, we have 
$$\liminf_{n\rightarrow \infty}\,n\widehat{\varphi}_n(p_c)\,=\,0\,.$$
\end{thm}
\noindent We shall
prove this theorem by contradiction. 
We first define an exploration process of the cluster of $0$ that remains inside the box $[-n,n]^d$. 
If the statement of the theorem does not hold, then the cluster of $0$ satisfies a $d$-dimensional anchored isoperimetric inequality. It follows that the number of sites that are revealed in the exploration of the cluster of $0$ will grow fast enough of order $n^{d-1}$.
Then, we can prove that the intersection of the cluster that we have explored with the boundary of the box $[-n,n]^d$ is of order $n^{d-1}$. Using the fact that there is no percolation in a half-space, we obtain a contradiction.
Before starting the precise proof, we
recall some results from \cite{DembinCheeger} on the meaning of the limiting value $\varphi(p)$. 

\vskip 0.5cm
\noindent
{\bf The Wulff theorem.} 
We denote by $\cL ^d$ the $d$-dimensional Lebesgue measure 
and by
$\cH ^{d-1}$ denotes the $(d-1)$--Hausdorff measure in dimension $d$.
Given a norm $\tau$ on $\sR^d$ and a subset $E$ of $\sR^d$ having a regular enough boundary, 
we define $\cI_\tau(E)$,
the surface tension of $E$ for the norm $\tau$, 
as 
$$\cI_\tau (E)=\int_{\partial E}\tau(n_E(x))\cH^{d-1}(dx)\,.$$
We consider
the anisotropic isoperimetric problem
associated with the norm $\tau$:
\begin{align}\label{isopb}
\text{minimize $\frac{\cI_\tau(E)}{\cL^d(E)}$ subject to $\cL^d(E)\leq 1$}\,.
\end{align}
The famous Wulff construction provides a minimizer for this anisotropic isoperimetric problem. We define the set $\widehat{W}_\tau$ as
$$\widehat{W}_\tau=\bigcap_{v\in\sS^{d-1}}\left\{x\in\sR^d:\, x\cdot v\leq \tau(v)\right\}\,,$$
where $\cdot$ denotes the standard scalar product and $\sS^{d-1}$ is the unit sphere of $\sR^d$.
Up to translation and Lebesgue negligible sets, the set
$$\frac{1}{\cL^d(\widehat{W}_\tau)^{1/d}}\widehat{W}_\tau$$ is the unique solution to the problem~\eqref{isopb}.
\vskip 0.5cm
\noindent
{\bf Representation of $\varphi(p)$.}
In \cite{DembinCheeger}, we build an appropriate norm $\beta_p$ for our problem that is directly related to the open edge boundary. We define the Wulff crystal $W_p$ as the dilate of $\widehat{W}_{\beta_p}$ such that $\cL^d(W_p)=1/\theta(p)$, where $\theta(p)=\Prb(0\in\sC_\infty)$. We denote by $\cI_p$ the surface tension associated with the norm $\beta_p$.  In \cite{DembinCheeger}, we prove that
$$\forall p>p_c(d)\qquad\varphi(p)=\cI_p(W_p)\,.$$

\section{Proofs}
We prove next the following lemma, which is based on two important results due to Zhang \cite{Zhang} and Rossignol and Théret \cite{flowconstant}. To alleviate the notation, the critical point $p_c(d)$ is denoted simply by $p_c$.
\begin{lem} 
	We have
	$$\lim_{\substack{p\rightarrow p_c\\p>p_c}}
	\,\,
	\Big(
	\theta(p)\delta_{\cI_p(W_p)}+(1-\theta(p))\delta_0\Big)
	\,=\,
	\delta_0\,.$$
\end{lem}
\begin{proof} 
If $\lim_{p\rightarrow p_c}\theta(p)=0$, then the result is clear. 
	Otherwise, let us assume that 
	$$\lim_{\substack{p\rightarrow p_c\\p>p_c}}\theta(p)=\delta>0\,.$$ Let $B$ be a subset of $\sR^d$ having a regular boundary and such that $\cL^d(B)=1/\delta$. As the map $p\mapsto \theta(p)$ is non-decreasing and $\cL^d(W_p)=1/\theta(p)$, we have $$\forall p>p_c\qquad\cL^d(W_p)\leq \cL^d(B)\,.$$ Moreover as $W_p$ is the dilate of the minimizer associated to the isoperimetric problem \eqref{isopb}, we have
$$\forall p>p_c\qquad\cI_p(W_p)\leq\cI_p(B)\,.$$
In \cite{Zhang}, Zhang proved that $\beta_{p_c}=0$. In \cite{flowconstant}, Rossignol and Théret proved the continuity of the flow constant. Combining these two results, we get that 
	$$\lim_{\substack{p\rightarrow p_c\\p>p_c}}\,
	\beta_p\,=\,\beta_ {p_c}\,=\,0\quad 
	\text{and so}\quad
	\lim_{\substack{p\rightarrow p_c\\p>p_c}}\,
	\cI_p(B)\,=\,0\,.$$
Finally, we obtain
	$$\lim_{\substack{p\rightarrow p_c\\p>p_c}}\,
	\cI_p(W_p)\,=\,0\,.$$
This yields the result.
\end{proof}

\begin{proof}[Proof of theorem \ref{heartp}]
We assume by contradiction that $$\Prb\left(\liminf_{n\rightarrow \infty}\,\,n\widehat{\varphi}_n(p_c)=0\right)<1\,.$$ Therefore there exist positive constants $c$ and $\delta$ such that
\begin{align}\label{refeq1}
\Prb\left(\liminf_{n\rightarrow \infty}\,n\widehat{\varphi}_n(p_c)>c\right)=\lim_{n\rightarrow\infty}\Prb\left(\inf_{k\geq n}\,k\widehat{\varphi}_k(p_c)>c\right)=\delta\,.
\end{align}
Therefore, there exists a positive integer $n_0$ such that
\begin{align}\label{refeq2}
\Prb\left(\inf_{k\geq n_0}\,\,k\widehat{\varphi}_k(p_c)>c\right)\geq\frac{\delta}{2}\,.
\end{align}
In what follows, we condition on the event 
	$$\Big\{\inf_{k\geq n_0}\,k\widehat{\varphi}_k(p_c)>c\,\Big\}\,.$$
Note that on this event, $0$ is connected to infinity by a $p_c$-open path. 
For $H$ a subgraph of $\sZ^d$,
we define
$$\partial ^o H=\Big\{\,e\in\partial H,\text{ $e$ is open}\,\Big\}\,.$$
Note that if $H\subset \sC_\infty$, then 
$\partial_{\sC_\infty} H=\partial^o H$. 
Moreover, if $H$ is equal to 
$\sC(0)$,
the open cluster of $0$, then
$\partial_{\sC(0)} H=\partial^o H=\varnothing$. 
We define next an exploration process of the cluster of $0$.
We set $\sC_0=\{0\}$, $\cA_0=\emptyset$. Let us assume that
	$\sC_0,\dots,\sC_l$ and $\sA_0,\dots,\sA_l$ are already constructed.  We define 
$$\sA_{l+1}=\left\{x\in\sZ^d\,:\,\exists y \in\sC_l\quad \langle x,y\rangle\in\partial^o \sC_l\right\}\,$$
and $$\sC_{l+1}=\sC_l\cup\sA_{l+1}\,.$$
We have
$$\partial ^o\sC_l\subset \{\langle x,y\rangle \in \E ^d :\, x\in \sA_{l+1}\}$$ so that $|\partial ^o\sC_l|\leq 2d |\sA_ {l+1}|$. Since $\sA_{l+1}$ and $\sC_l$ are disjoint, we have
\begin{align}\label{iter}
|\sC_{l+1}|=|\sC_l|+|\sA_ {l+1}|\geq |\sC_l|+\frac{|\partial ^o\sC_l|}{2d}\,.
\end{align}
Let us set $\alpha=1/n_0^d$ so that $|\sC_0|=\alpha n_0^d$. Let $k$ be the smallest integer greater than $2^{d+1}d/c$. We recall that $c$ and $n_0$ were defined in \eqref{refeq1} and \eqref{refeq2}.
	Let us prove by induction on $n$ that 
\begin{equation}\label{ire}
\forall n\geq n_0\qquad|\sC_{(n-n_0)k}|\geq \alpha n^d\,.
\end{equation}
This is true for $n=n_0$. Let us assume that this inequality is true for some integer $n\geq n_0$. 
	If $|\sC_{(n+1-n_0)k}|\geq n^d$, then we are done. 
	Suppose that $|\sC_{(n+1-n_0)k}|< n^d$. 
	In this case, 
	for any integer $l\leq k$, we have also
$|\sC_{(n-n_0)k+l}|<n^d$, and since
	$\sC_{(n-n_0)k+l}$ is a valid subgraph of $\sC(0)$ and
$\widehat{\varphi}_{n}(p_c)>c/n$, we conclude that
	$$\frac{|\partial ^o \sC_{(n-n_0)k+l}|}{|\sC_{(n-n_0)k+l}|}\geq \frac{c}{n}$$
and so $|\partial ^o \sC_{(n-n_0)k+l}|\geq \alpha c n ^{d-1}$.
Thanks to inequality \eqref{iter} applied $k$ times, we have 
$$|\sC_{(n+1-n_0)k}|\geq \alpha \left(n^d+\frac{ck}{2d}n^{d-1}\right)\,.$$
As $k \geq 2^{d+1}d/c$, we get
$$|\sC_{(n+1-n_0)k}|\geq \alpha (n^d+2^dn^{d-1})\geq \alpha (n+1) ^{d}\,.$$
This concludes the induction.

Let $\eta>0$ be a constant that we will choose later. In \cite{BarskyGrimmettNewman}, Barsky, Grimmett and Newman proved that there is no percolation in a half-space at criticality.
An important consequence of the result of Grimmett and Marstrand \cite{GrimmettMarstrand} is that the critical value for bond percolation in a half-space equals to the critical parameter $p_c(d)$ of bond percolation in the whole space, \textit{i.e.}, we have 
$$\Prb(0\text{ is connected to infinity by a $p_c$-open path in $\sN\times\sZ^{d-1}$})=0\,,$$
so that for $n$ large enough,
$$\Prb(\text{$\exists\gamma$ a $p_c$-open path starting from $0$ in $\sN\times\sZ^{d-1}$ such that $|\gamma|\geq n$})\leq \eta\,.$$
In what follows, we will consider an integer $n$ such that the above inequality holds. By construction the set $\sC_{n}$ is inside the box $[-n,n]^d$. Starting from this cluster, we are going to resume our exploration but with the constraint that we do not explore anything outside the box $[-n,n]^d$.
We set $\sC'_0=\sC_{n}$ and $\cA'_0=\emptyset$. Let us assume $\sC'_0,\dots,\sC'_l$ and $\sA'_0,\dots,\sA'_l$ are already constructed.  We define 
$$\sA'_{l+1}=\big\{\,x\in[-n,n]^d\,:\,\exists y \in\sC'_l\quad \langle x,y\rangle\in\partial^o \sC'_l\,\big\}\,$$
and $$\sC'_{l+1}=\sC'_l\cup\sA'_{l+1}\,.$$
We stop the process when $\sA'_{l+1}=\emptyset$. As the number of vertices in the box $[-n,n]^d$ is finite, this process of exploration will eventually stop for some integer~$l$. We have that $|\sC'_l|\leq n^d$ and $n\hat{\varphi}_k(p_c)>c$ so that
$$|\partial ^o \sC'_l|\,\geq\, \frac{c}{n}|\sC'_l|\,\geq\, \frac{c}{n}|\sC_n|\,.$$
	Moreover, for $n\geq  k n_0$, we have, thanks to inequality~\eqref{ire},
$$|\sC_n|\,\geq\,
\big|\sC_{\lfloor\frac{n}{k}\rfloor k}\big|\,\geq\,
\big|\sC_{(\lfloor\frac{n}{k}\rfloor -n_0)k}\big|\,\geq\,
\alpha\Big(\Big\lfloor\frac{n}{k}\Big\rfloor \Big)^d
\,.$$
We suppose that $n$ is large enough so that 
	$n\geq  k n_0$ and
$\lfloor\frac{n}{k}\rfloor\geq  n/2k$. 
Combining the two previous display inequalities, we conclude that
$$|\partial ^o \sC'_l|\,\geq\,\frac{c\alpha
}{2^dk ^d  }
n ^{d-1}
\,.$$
Therefore, for $n$ large enough,
there exists one face of $[-n,n]^d$ such that
there are at least $c\alpha n ^{d-1}/(2^dk ^d2d )$ vertices that are connected to $0$ by a $p_c$-open path that remains inside the box $[-n,n]^d$ and so
\begin{align}\label{eqabs}
\Prb\left(\begin{array}{c}
 \text{ there exists one face of $[-n,n]^d$ with
at least}\\\text{ $c\alpha n ^{d-1}/(2^dk ^d 2d)$ vertices that are connected to $0$ by a }\\\text{$p_c$-open path that remains inside the box $[-n,n]^d$}\end{array}\right)\geq \frac{\delta}{2}\,.
 \end{align}
Let us denote by $X_n$ the number of vertices in the face $\{-n\}\times [-n,n]^{d-1}$ that are connected to $0$ by a $p_c$-open path inside the box $[-n,n]^d$. We have
\begin{align}\label{eq1} 
 \E(X_n)&\leq \left|(\{-n\}\times [-n,n]^{d-1})\cap\sZ ^d \right|\,\Prb\left(\begin{array}{c}\text{$\exists\gamma$ a $p_c$-open path starting}\\\text{ from $0$ in $\sN\times\sZ^{d-1}$ such that}\\ |\gamma|\geq n\end{array}\right)\nonumber\\
 &\leq  (2n+1)^{d-1}\eta\,.
\end{align} 
Moreover, we have
\begin{align} \label{eq2}
 \E(X_n)\,\geq\,
	\frac{c\alpha}{2d2^dk ^d  }n ^{d-1}\,\, \Prb\left(X_n>\frac{c\alpha}{2d2^dk ^d  }n ^{d-1}\right)\,.
 \end{align}
Finally, combining inequalities \eqref{eq1} and \eqref{eq2}, we get
$$\Prb\left(X_n>\frac{c\alpha}{2d2^dk ^d  }n ^{d-1}\right)\leq \frac{2d\eta 3^{d-1}2^dk^d}{c\alpha}\,.$$
Therefore, we can choose $\eta$ small enough such that 
$$\Prb\left(X_n>\frac{c\alpha}{2d2^dk ^d  }n ^{d-1}\right)\leq \frac{\delta}{10d}$$
and so using the symmetry of the lattice
\begin{align*}
\Prb&\left(\begin{array}{c}
 \text{ there exists one face of $[-n,n]^d$ such there are at least}\\\text{ $c\alpha n ^{d-1}/(2^dk ^d 2d)$ vertices that are connected to $0$ by a $p_c$-open}\\\text{ path that remains inside the box $[-n,n]^d$}\end{array}\right)\nonumber\\
 &\leq 2d\,\Prb\left(X_n>\frac{c\alpha}{2d2^dk ^d  }n ^{d-1}\right)\leq\frac{\delta}{5}\,.
 \end{align*}
 This contradicts inequality \eqref{eqabs} and yields the result.
\end{proof}

\end{document}